\documentclass[11pt]{article}

\usepackage{amsxtra}
\usepackage{epsfig,tabularx}
\usepackage{amssymb,amsmath,amsthm,amsgen,amsxtra}
\usepackage{booktabs}
\usepackage{url}

\makeatletter
\newfont{\footsc}{cmcsc10 at 8truept}
\newfont{\footbf}{cmbx10 at 8truept}
\newfont{\footrm}{cmr10 at 10truept}
\renewcommand{\ps@plain}{%
\renewcommand{\@oddfoot}{\footsc 
\hfil\footrm\thepage}}
\makeatother
\theoremstyle{plain}
\newtheorem{theorem}{Theorem}
\newtheorem{lemma}[theorem]{Lemma}

\newtheorem{conjecture}[theorem]{Conjecture}
\theoremstyle{definition}
\newtheorem{fact}[theorem]{Fact}

\newtheorem{remark}[theorem]{Remark}

\newcommand{\sumsb}[1]{\sum_{\substack{#1}}}  

\newcommand{\muset}{M}
\newcommand{\nkmin}[2]{(#1,#2)}

\newcommand{\EPn}[1]{EP(#1)}
\newcommand{\uncEPn}[1]{EP_{\text{unc}}(#1)}
\newcommand{\EPmun}[1]{EP_{\mu > 0}(#1)}
\newcommand{\mup}[1]{EP'_{\mu>0}(#1)}
\newcommand{\mupall}{EP'_{\mu>0}}

\newcommand{\irrmu}[1]{\text{Irr}(#1)}

\newcommand{\pww}{P_{w,w}(q)}
\newcommand{\puv}{P_{u,v}(q)}
\newcommand{\pxw}{P_{x,w}(q)}

\newcommand{\pxws}{P_{x,ws}(q)}
\newcommand{\pxsws}{P_{xs,ws}(q)}

\newcommand{\psxw}{P_{sx,w}(q)}
\newcommand{\pxsw}{P_{xs,w}(q)}
\newcommand{\pxz}{P_{x,z}(q)}

\newcommand{\bbZ}{\mathbb{Z}}

\newcommand{\simref}{\mathcal{S}}

\DeclareMathOperator{\rds}{rds}
\DeclareMathOperator{\lds}{lds}

\newcommand{\digw}{\mathcal{D}_w}
\newcommand{\digx}{\mathcal{D}_x}
\newcommand{\dxw}{d_{x,w}}

\newcommand{\simls}{\sim_{ls}}

\newcommand{\simlspm}{\sim}
\newcommand{\simlspmk}[1]{\overset{#1}{\simlspm}}

\newcommand{\ec}[2]{[[#1,#2]]}

\newcommand{\ecls}[2]{[[#1,#2]]_{ls}}

\title{Equivalence classes for the $\mu$-coefficient of
  Kazhdan-Lusztig polynomials in $S_n$\footnote{Supported in part by
    National Security Agency grant H98230-09-1-0023.}}
 
\author{Gregory S. Warrington \\
University of Vermont \\
Burlington, VT 05401\\
gwarring@cems.uvm.edu}

\begin{document} 
\maketitle

\begin{abstract}
  We study equivalence classes relating to the Kazhdan-Lusztig
  $\mu(x,w)$ coefficients in order to help explain the scarcity of
  distinct values.  Each class is conjectured to contain a
  ``crosshatch'' pair.  We also compute the values attained by
  $\mu(x,w)$ for the permutation groups $S_{10}$ and $S_{11}$.
\end{abstract}


\section{Introduction}
\label{sec:intro}

The Kazhdan-Lusztig polynomials, introduced in~\cite{K-L1}, arose in
the context of constructing representations of the Hecke algebra
associated to a Weyl group.  It was soon apparent that these
polynomials encode important information relating to geometry and
representation theory.  For example, they encode the singularities of
Schubert varieties and the multiplicities of irreducibles in Verma
modules~\cite{BeilBern,BryKash,K-L1}.  They are also of interest from
a purely combinatorial viewpoint (see ~\cite{BjornerBrenti}).

We restrict our attention to the type-$A$ case in which there is one
Kazhdan-Lusztig polynomial $\pxw$ associated to every pair of
permutations $x,w\in S_n$.  Kazhdan and Lusztig give a simple
recursion for these polynomials in their original paper (see
Section~\ref{sec:kl} below).  However, our combinatorial understanding
of these polynomials is still far from complete.  For example, there
is neither a \emph{combinatorial} proof that the coefficients of
$\pxw$ are nonnegative nor a closed formula for the degree of a given
polynomial.  (A non-combinatorial proof of nonnegativity arises from
the interpretation of the coefficients of Kazhdan-Lusztig polynomials
in terms of intersection cohomology~\cite{K-L2}.)  The reason these
problems are still open is that there is a correction term in the
recursion that is controlled by a poorly understood number,
$\mu(x,w)$.  While there are known to be a few simple, combinatorial
\emph{necessary} conditions for $\mu(x,w)$ to be nonzero, these
conditions are by no means sufficient.  In fact, there are no
nontrivial sufficient conditions known.  A combinatorial rule for the
value $\mu(x,w)$ would likely lead to insights wherever
Kazhdan-Lusztig polynomials arise.

A major difficulty in the study of these $\mu$-coefficients is that
(as shown in~\cite{01conj}) $S_{10}$ is the smallest symmetric group
for which $\mu(x,w)$ can be anything other than $0$ or $1$.  There is
little overlap between what is computationally feasible and what is
computationally illuminating.  Nonetheless, there are a number of
important combinatorial results regarding these polynomials.  See the
book by Bj\"orner and Brenti~\cite{BjornerBrenti} for an overview and
the papers of Brenti (such as~\cite{brentilat} and~\cite{brentiinv})
in particular.

The organization of the paper is as follows.  Section~\ref{sec:dfns}
provides the necessary definitions while Section~\ref{sec:props}
outlines the properties of $\mu(x,w)$ we will be using from the
literature.  The results of this paper are of two types.  First, we
present new data regarding the values $\mu(x,w)$ takes; how we do this
is outlined in Section~\ref{sec:mup}.  Set $\muset(n) = \{\mu(x,w):\,
x,w\in S_n\}\setminus\{0\}$.
\begin{theorem}\label{thm:numval}
  We have
  \begin{itemize}
  \item $\muset(10) = \{1,4,5\}$,
  \item $\muset(11) = \{1,3,4,5,18,24,28\}$ and
  \item $\muset(12) \supseteq \{1,2,3,4,5,6,7,8,18,23,24,25,26,27,28,158,163\}$.
  \end{itemize}
\end{theorem}
Particular pairs $x,w$ realizing each of these values are given in
Table~\ref{table:klpolys}.  The only $\mu$-values that have
already appeared in the literature for $S_n$ are $\{0,1,2,3,4,5\}$.

We also offer computer code~\cite{mycode} that can quickly produce a
database of all Kazhdan-Lusztig polynomials in $S_{10}$; this code is
discussed in Section~\ref{sec:klcomp}.  There are over one billion
``extremal pairs'' $(x,w)$ in $S_{10}$ for which one might hope
$\mu(x,w) > 0$.  More than 100 million of these pairs cannot be
reduced to equivalent pairs in smaller symmetric groups.  Altogether,
approximately one million different polynomials appear.  Even stored
efficiently this yields a gigabyte of data.  The comparable database
for $S_{11}$ would be on the order of $50$ times larger.

Second, we consider the question of why there are so few different
values of $\mu(x,w)$.  For example, in $S_{10}$ there are $664\,752$
non-covering pairs $x < w$ for which $\mu(x,w) > 0$.  Yet, the only
nonzero values taken are $1$, $4$ and $5$.  We explain this in
Section~\ref{sec:ecs} by showing that for $S_{10}$ and $S_{11}$, the
$\mu$-positive pairs fall into a handful of equivalence classes.  The
$\mu$-coefficient is constant on each class by construction.  The
equivalence relation, $\simlspm$, is defined in Section~\ref{sec:ecs};
the corresponding class of a pair $(x,w)$ is denoted $\ec{x}{w}$.  A
class is \emph{$n$-minimal} if it does not intersect $S_m$ for $m <
n$.  Pairs in $n$-minimal classes are also referred to as $n$-minimal
themselves.  As a consequence of Theorem~\ref{thm:numval}, the number
of $10$- and $11$-minimal classes is at least $2$ and $4$,
respectively.
\begin{theorem}\label{thm:numequiv}
  The $2$-minimal class $\ec{01}{10}$ is the only class intersecting
  $S_m$ for any $m < 10$.  The number of $10$- and $11$-minimal
  classes is at most $4$ and $7$, respectively.
\end{theorem}

Finally, in Section~\ref{sec:reps} we speculate that each
$\simlspm$-equivalence class contains a ``crosshatch'' pair.

\section{Definitions}\label{sec:dfns}
\subsection{The symmetric group}
\label{sec:sn}

The symmetric group, $S_n$, has the following presentation as a
Coxeter group:
\begin{equation}
\begin{aligned}
  S_n = \langle s_1,\ldots,s_{n-1}\, :\, &s_i^2 = 1, \\
  &s_is_{i\pm 1}s_i = s_{i\pm 1}s_is_{i\pm 1}, \text{ and }\\
  &s_is_j = s_js_i, \text{ for } |i-j| > 1\rangle.
\end{aligned}
\end{equation}
We write $\simref$ for the set of generators $\{s_1,\ldots,s_{n-1}\}$.
The group $S_n$ is often described as the group of bijections from
$\{0,1,\ldots,n-1\}$ to itself (i.e., permutations) under the usual
function composition.  From this perspective, it is most convenient to
identify the generator $s_i$ with the adjacent transposition that
switches $i-1$ and $i$.  For clarity in examples, we will write $a$
for $10$, $b$ for $11$, etc.  \emph{One-line notation} for $\sigma \in
S_n$ lists the elements $[\sigma(0),\sigma(1),\ldots,\sigma(n-1)]$ in
order.  We often omit commas and brackets.  For example, the
permutation $\sigma \in S_6$ that sends $i$ to $5-i$ would either be
written $[5,4,3,2,1,0]$ or simply $543210$.

The group $S_n$ has the structure of a ranked poset as follows.  An
\emph{inversion} of a permutation $w = [w(0),w(1),\ldots,w(n-1)]$ is a
pair $i < j$ for which $w(i) > w(j)$.  The \emph{length} of $w$,
$\ell(w)$, is the total number of inversions.  The rank of an element
is then given by its length.  To define the partial order under which
we will be relating our elements, we first make two auxiliary
definitions.  Let $x,w \in S_n$ and $p,q\in \bbZ$.  Define $r_w(p,q) =
|\{i \leq p: w(i) \geq q\}|$ and the \emph{difference function}
$\dxw(p,q) = r_w(p,q) - r_x(p,q)$.  Then the \emph{Bruhat partial
  order}, $\leq$, is determined by setting $x\leq w$ if and only if
$\dxw(p,q) \geq 0$ for all $p,q$.  This definition is equivalent to
more common ones such as the tableau criterion
(cf.~\cite{BLak,Fulton-book,Hum}).  

For a permutation $w$, let $\digw$ denote the permutation matrix
oriented such that for each $i$ there is a 1 in the $i$-th column from
the left and $w(i)$-th row from the bottom.  We will frequently
display a pair of permutations $x$ and $w$ graphically using
\emph{Bruhat pictures}: Such a picture consists of $\digw$ and $\digx$
overlaid along with shading given by the difference function.  An
example is given in Figure~\ref{fig:brupic}.  Entries of $\digx$
(resp., $\digw$) are denoted by black disks (resp., circles).
Positions corresponding to $1$s of both $\digx$ and $\digw$ (termed
\emph{capitols}) are denoted by a black disk and a larger, concentric
circle.  Shading denotes regions in which $\dxw \geq 1$.  Successively
darker shading denotes successively higher values of $\dxw$.
\begin{figure}[htbp]
  \centering
      {\scalebox{.6}{\includegraphics{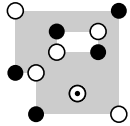}}}
      \caption{Bruhat picture for $x = [2,0,4,1,3,5]$, $w = [5,2,3,1,4,0]$.}
      \label{fig:brupic}
\end{figure}

Finally, there are two sets we associate to any permutation $w$.  We
define the \emph{right descent set} of $w$, $\rds(w)$, as $\{s\in
\simref: ws < w\}$.  Similarly, the \emph{left descent set} is
$\lds(w) = \{s\in\simref: sw < w\}$.

\subsection{Kazhdan-Lusztig polynomials}
\label{sec:kl}

We now define the Kazhdan-Lusztig polynomials $\pxw$ associated to
pairs of elements $x,w\in S_n$.  For motivation and more general
definitions applicable to any Coxeter group, we refer the reader
to~\cite{Hum,K-L1}.  Set
\begin{equation}\label{eq:mudef}
  \mu(x,w) = \text{ coefficient of }
             q^{(\ell(w)-\ell(x)-1)/2}\text{ in }\pxw
\end{equation}
and define $c_s(x) = 1$ if $xs < x$; $c_s(x) = 0$ if $xs > x$.  We
have the following paraphrased theorem of Kazhdan and Lusztig:
\begin{theorem}[\cite{K-L1}]\label{thm:kldef}
  There is a unique set of polynomials $\{\pxw\}_{x,w\in S_n}$ such
    that, for $x,w\in S_n$:
  \begin{itemize}
  \item $\pww = 1$,
  \item $\pxw = 0$ when $x\not\leq w$ and
  \item for $s\in\rds(w)$, 
    \begin{equation}\label{eq:std}
      \pxw = q^{c_s(x)}\pxws + q^{1-c_s(x)}\pxsws - \sumsb{z \leq
        ws\\ zs < z} \mu(z,ws)q^{\frac{\ell(w)-\ell(z)}{2}} \pxz.      
    \end{equation}
  \end{itemize}
  When $x < w$ we have an upper bound on the degrees: $\deg(\pxw) \leq
  \frac{\ell(w)-\ell(x)-1}{2}$.
\end{theorem}
Note that $\mu(x,w)$ is the coefficient of the highest possible power
of $q$ in $\pxw$.

\section{Properties satisfied by $\mu(x,w)$}
\label{sec:props}

We now proceed to describe various well-known properties satisfied by
the $\mu$-coefficient.  If $x\not\leq w$, then $\mu(x,w)$ is
automatically zero.  So assume $x \leq w$.  There are two
easily recognized instances in which the $\mu$-coefficient is zero.
The first follows directly from the definitions since $P_{x,w}$ is a
polynomial in $q$ rather than $q^{1/2}$.
\begin{fact}\label{fact:even}
  If $\ell(w)-\ell(x)$ is even, then $\mu(x,w) = 0$.
\end{fact}
We will refer to a pair $x,w$ for which which $\ell(w)-\ell(x)$ is
odd as an \emph{odd pair}.

The second follows from an important set of equalities satisfied by
the Kazhdan-Lusztig polynomials (see~\cite[Corollary 7.14]{Hum} for a
proof):
\begin{equation}\label{eq:xs}
  \pxw = \pxsw \text{ if } s\in\rds(w)\text{ and }
  \pxw = \psxw \text{ if } s\in\lds(w).
\end{equation}
Define the set of \emph{extremal pairs}
\begin{align}
  \EPn{n} = \{x\leq w\in S_n\times S_n:\, \lds(x)\supseteq
  \lds(w)\text{ and }\rds(x) \supseteq \rds(w)\}.
\end{align}

\begin{fact}\label{fact:notep}
  If $\ell(x)<\ell(w)-1$ and $(x,w)\not\in\EPn{n}$, then $\mu(x,w) = 0$.
\end{fact}
To see why Fact~\ref{fact:notep} is true, suppose we have a
non-covering pair $x < w$ along with some $s\in \simref$ such that $xs
> x$ and $ws < w$.  The equality $\pxw = \pxsw$ combined with the
degree bound of Theorem~\ref{thm:kldef} implies, since
$\ell(w)-\ell(xs) = \ell(w)-\ell(x)-1$, that the coefficient of
$q^{(\ell(w)-\ell(x)-1)/2}$ in $\pxw$ must be zero.

According to computations in~\cite{Hammett}, there are approximately
800 billion comparable pairs $x,w$ in $S_{10}$.  It turns out that whenever
$w$ covers $x$, $\pxw = \mu(x,w) = 1$; ignore these pairs for the
moment.  Then, considering only pairs for which $\mu(x,w) > 0$,
Facts~\ref{fact:even} and~\ref{fact:notep} allow us to restrict our
attention to the odd extremal pairs.  The number of such pairs in
$S_{10}$ is a modest 626\,145\,374, yet still much larger than
$|\muset(10)|=3$.

The idea of considering equivalence classes to explain the redundancy
of $\mu$-values is not new.  Lascoux and Sch\"utzenberger, and
probably others, entertained the possibility that any pair $x,w$ with
$\mu(x,w) > 0$ could be generated from a cover by applying certain
operators (see the L-S operators below).  By construction, all pairs
generated in this way would have the same $\mu$-value.  Our main
contribution in this paper in this regard is to consider
``compression'' (and ``decompression'') \emph{in conjunction with} the
L-S operators and symmetry.  Our hope is that these classes are large
enough to fully explain the scarcity of distinct values of $\mu$.  The
three relations from which we build these classes exist already in the
literature.  We now describe them.

The simplest relations (of various symmetries) can be derived from the
definitions in~\cite{K-L1}.
\begin{fact}\label{fact:symm}
  Let $w_0$ denote the \emph{long word} $[n-1,n-2,\ldots,1,0]$ in
  $S_n$.  Then for $x,w\in S_n$,
  \begin{equation}
    \mu(x,w) = \mu(x^{-1},w^{-1}) = \mu(w_0w,w_0x) = \mu(ww_0,xw_0).
  \end{equation}
\end{fact}
Our second relation arises from the \emph{Lascoux-Sch\"utzenberger
  (L-S) operators} (which, their name notwithstanding, were known to
Kazhdan and Lusztig~\cite{K-L1}).  Define $\mathcal{R}_k$ be the set
of permutations $w$ for which $ws_k < w$ or $ws_{k+1} < w$, but not
both.  In other words, $\mathcal{R}_k$ consists of all permutations in
which $w(k)$, $w(k+1)$, $w(k+2)$ do \emph{not} appear in increasing or
decreasing order.  Then $wR_k$ is defined to be the unique element in
the intersection $\mathcal{R}_k \cap \{ws_k,ws_{k+1}\}$.  The
operators $R_k$ act ``on the right'' in the sense that they act on
positions.  Operators $L_k$ that act ``on the left'' can be defined
analogously by having them act on values.  More precisely, we set
$\mathcal{L}_k = \{w:\, w^{-1}\in \mathcal{R}_k\}$ and
$L_kw=(w^{-1}R_k)^{-1}$.  (These operators, elementary Knuth
transformations and their duals, are closely connected to the
Robinson-Schensted correspondence; for details, see
\cite{Fulton-book,Knuth70}.)  For $x,w\in S_n$, set
  \begin{equation}\label{eq:mus}
    \mu[x,w] = 
    \begin{cases}
      \mu(x,w), & \text{ if } x \leq w,\\
      \mu(w,x), & \text{ if } w \leq x,\\
      0, & \text{ if $x$ and $w$ are not comparable}.
    \end{cases}
  \end{equation}

\begin{fact}[\cite{K-L1}]\label{fact:samel}
  If $x,w\in \mathcal{L}_k$, then $\mu[x,w] = \mu[L_kx,L_kw]$.
  If $x,w\in \mathcal{R}_k$, then $\mu[x,w] = \mu[xR_k,wR_k]$.\\
\end{fact}
Note that the L-S operators \emph{do not} preserve the lower-order
coefficients of Kazhdan-Lusztig polynomials.  Also note that
$\mu(\cdot,\cdot)$ is not invariant under the L-S operators (consider
$L_0$ acting on the pair $(021,201)$).  In the rest of the paper, when
we refer to $\mu$ being constant on an equivalence class, we are
referring to $\mu[\cdot,\cdot]$ rather than $\mu(\cdot,\cdot)$.

\begin{figure}[htbp]
  \centering
      {\scalebox{.6}{\includegraphics{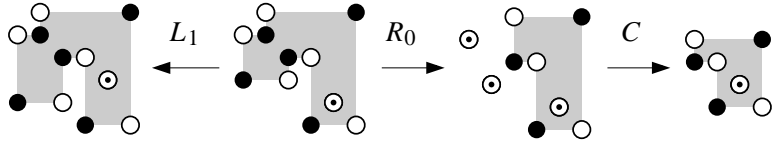}}}
      \caption{Example actions of $L_1$ and $R_0$ on the pair
        $x=243015$, $w=452310$.  The simultaneous compression of
        $(xR_0,wR_0)$ at two capitols is displayed in the rightmost
        figure.}
      \label{fig:lsex}
\end{figure}

Our third relation, unlike the L-S operators, has the potential to
take a pair in one symmetric group into a pair in a \emph{different}
symmetric group.

We say that a capitol for a pair $x,w\in S_n$ is \emph{naked} if it
lies within an unshaded region of the corresponding Bruhat picture.
The \emph{compression}, $(x^{\hat{\imath}},w^{\hat{\imath}})$, of
$(x,w)$ at the naked capitol $(i,x(i)) = (i,w(i))$ corresponds to
deleting the $i$-th columns and $w(i)$-th rows of $\digx$ and $\digw$.
Running the process in reverse is termed a \emph{decompression}.  The
pair $(x,w)$ is \emph{uncompressible} if its Bruhat picture has no
naked capitols.  Note that compressing a pair $x,w\in S_n$ produces a
pair in $S_{n-1}$ while decompression produces one in $S_{n+1}$.  In
figures, compression(s) will be denoted by a ``$C$'' and decompressions
by a ``$D$.''  A proof of the following can be found in~\cite[Lemma
39]{gwsb-msl}.

\begin{fact}\label{fact:compress}
  For any naked capitol $(i,x(i)) = (i,w(i))$, both $\pxw =
  P_{x^{\hat{\imath}},w^{\hat{\imath}}}$ and $\ell(w)-\ell(x) =
  \ell(w^{\hat{\imath}})-\ell(x^{\hat{\imath}})$.  Hence, $\mu(x,w) =
  \mu(x^{\hat{\imath}},w^{\hat{\imath}})$.
\end{fact}

\section{Results}

\subsection{Computation of Kazhdan-Lusztig polynomials}
\label{sec:klcomp}

Construction of the database encoding all Kazhdan-Lusztig polynomials
for pairs $x,w\in S_m$ with $m\leq 10$ proceeded by a direct
application of~\eqref{eq:std}.  Our algorithm is basically that of the
original recursion of Kazhdan and Lusztig~\cite{K-L1} as described
in~\cite{Hum}.  However, two aspects of our algorithm merit note.

First, equation~\eqref{eq:xs} allows us to focus on extremal pairs.
As in du~Cloux's program~\cite{ducloux}, when required to compute
$\pxw$ for any pair $(x,w)\not\in \EPn{n}$, we simply move $x$ up in
the Bruhat order through the action of elements of $\rds(w)$ and
$\lds(w)$.  Second, Fact~\ref{fact:compress} allows us to focus on
uncompressible pairs.  When required to compute the Kazhdan-Lusztig
polynomial for a compressible pair, we take the novel approach of
first compressing as much as possible to a pair $(x',w')$.  Often,
this resulting pair is not extremal.  Moving $x'$ up in the Bruhat
order can then lead to additional naked capitols.  The process can
repeat as illustrated in Figure~\ref{fig:rinserepeat}.

\begin{figure}[htbp]
  \centering
  {\scalebox{.6}{\includegraphics{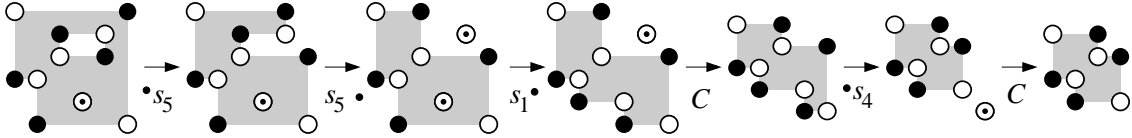}}}
  \caption{Example of how compression can lead to an extremal pair no
    longer being extremal.}
  \label{fig:rinserepeat}
\end{figure}

A great deal of redundancy is avoided by only keeping track of the
uncompressible extremal pairs.  In $S_{10}$, for example, 90 percent
of the extremal pairs are compressible.

Table~\ref{table:kldata} collects various data regarding
Kazhdan-Lusztig polynomials and their computation.  The first five
rows list the number of extremal pairs, uncompressible extremal pairs,
extremal pairs with positive $\mu$-value, irreducible pairs and
$\nkmin{n}{0}$-minimal pairs, respectively (these last two terms are
defined in Sections~\ref{sec:mup} and~\ref{sec:ecs}, respectively).  The final
two rows reflect (among \emph{all} $\pxw$ with $x,w\in S_n$) the
maximum coefficient encountered and the number of distinct,
non-constant polynomials appearing, respectively.  Due to memory
constraints, we have only partial results for $S_{11}$.

\begin{table}[ht]
  \centering
  \caption{Kazhdan-Lustzig data for various $S_n$.}
  \begin{tabular}{@{}lrrrrrrrr@{}} \toprule
    $n$ & 4 & 5 & 6 & 7 & 8 & 9 & 10 & 11\\\midrule
    $|\EPn{n}|$ & 6 & 122 & 2\,220 & 45\,184 & 1\,107\,636 & 33\,487\,176 & 1\,248\,544\,230 & 56\,786\,656\,838\\
    $|\uncEPn{n}|$ & 2 & 10 & 152 & 3\,114 & 84\,624 & 2\,896\,168 & 122\,345\,174 & 6\,252\,533\,464\\
    $|\EPmun{n}|$ & 2 & 2 & 30 & 176 & 2\,312 & 33\,550 & 664\,752 & \\
    $|\irrmu{n}|$ & 0 & 0 & 0 & 0 & 0 & 16 & 2\,663 & 54\,214 \\
    $|\nkmin{n}{0}\text{-minimal}|$ & 0 & 0 & 0 & 0 & 0 & 12 & 2\,512 & 51\,060 \\\midrule
    max coeff. & 1 & 2 & 4 & 15 & 73 & 460 & 4\,176 & $\geq$18\,915\\
    $|\{\pxw\}|$ & 1 & 4 & 16 & 97 & 1\,118 & 24\,361 & 981\,174
  \end{tabular}\label{table:kldata}
\end{table}

\begin{remark}\label{remark:mem}
  It is not clear how to fully take advantage of parallel computers in
  the computation of collections of Kazhdan-Lusztig polynomials via
  equation~\eqref{eq:std}.  The computation of $\pxw$ is not local in
  the sense that it is not clear which $\puv$ will be required during
  the recursive steps.  In fact, due to the structure of the recursive
  branching, any given $\puv$ may be required \emph{many} times.  As
  such, the most efficient approach appears to store the intermediate
  $\puv$ whenever possible.  For $S_{11}$, however, such a database
  (useful in this way only if kept in RAM) would run roughly 50
  gigabytes.
\end{remark}

\subsection{Computing possible $\mu$-values}
\label{sec:mup}

For $n \leq 10$, the possible $\mu$-values can be extracted directly
from the database.  For $n=11$, the memory constraints discussed in
Remark~\ref{remark:mem} prevented us from computing the
Kazhdan-Lusztig polynomials for all uncompressible extremal pairs.
Fortunately, the identities of Section~\ref{sec:props} provide a
simple way to filter out pairs $x,w$ for which $\mu(x,w)\not\in
\muset(n)\setminus\muset(n-1)$.

Define two pairs in $S_n$ to be \emph{$\simls$-equivalent} if they can
be connected by via a finite chain of L-S operators.  Denote the
corresponding equivalence classes by $\ecls{x}{w}$.  Let $x,w$ be an
odd pair.  Suppose $\ecls{x}{w}$ contains a pair $u,v$ that is either
\begin{enumerate}
  \item compressible, 
  \item not extremal and with $\ell(u) < \ell(v)-1$, or
  \item not related in the Bruhat order.
\end{enumerate}
In the first case, $\mu(x,w) \in \muset(m)$ for some $m < n$.  But the
following lemma already tells us that such values are contained in
$\muset(n)$.
\begin{lemma}
  For $n\geq 2$, $\muset(n-1)\subseteq \muset(n)$.
\end{lemma}
\begin{proof}
  Any pair $x,w\in S_{n-1}$ can be decompressed by adding a capitol in
  the $n$-th row and $n$-th column.  The lemma then follows by
  Fact~\ref{fact:compress}.
\end{proof}
In the second and third cases, $\mu(x,w)$ must be $0$.  So, in looking for
elements of $\muset(n)\setminus \muset(n-1)$, we can restrict our
attention to odd extremal pairs in $S_n$ for which none of the three
above cases apply.  Such pairs will be termed \emph{irreducible}.  It
is significantly faster to compute whether a pair is irreducible than
it is to compute the corresponding Kazhdan-Lusztig polynomial.

Even though there are over half a million $\mu$-positive pairs in
$S_{10}$, there are only $2\,663$ irreducible pairs.  The computation
of the Kazhdan-Lusztig polynomials for the $54\,214$ irreducible
pairs in $S_{11}$ can be done in a few thousand hours of CPU time.

This completes the description of the worked required for the first
two parts of Theorem~\ref{thm:numval}.  The elements of $\muset(12)$
given there stem from individual Kazhdan-Lusztig polynomials we chose
to compute guided by Conjecture~\ref{conj:main}.  See
Table~\ref{table:klpolys} for representative pairs yielding these
$\mu$-values.  (In the table, the polynomial $a_0 + a_1q + a_2q^2 +
\cdots$ is described by its coefficient list: $a_0,a_1,a_2,\ldots$.)

\begin{table}[ht]
  \centering
  \caption{Known values of $\mu(x,w)$ and pairs that achieve them.}
  \begin{tabular}{@{}lrccl@{}} \toprule
    $n$ & $\mu$ & $x$ & $w$ & $\pxw$\\\midrule
    1 & 1 & 01 & 10 & 1 \\\midrule
    10 & 4 & 0432187659 & 4678091235 &  1,14,60,96,43,4\\
    & 5 & 2106543987 & 5678901234 &  1,10,43,86,84,37,5\\\midrule
    11 & 3 & 108765432a9 & 789a4560123  & 1,14,82,247,420,420,235,60,3\\
    & 18 & 21076543a98 & 792a4560813 &  1,16,112,442,1038,1485,1309,698,200,18\\
    & 24 & 1065432a987 & 689a1345702 &  1,17,129,556,1416,2143,1919,993,269,24\\
    & 28 & 21076543a98 & 6789a123450 &  1,18,145,646,1654,2516,2283,1197,325,28\\\midrule
    12 & 6 & 107654328ba9 & b6789a123450 & 1,24,267,1772,7554,21518,41845,55849,\\
    &    &              &              &    \quad 50705,30547,11637,2552,259,6\\
    & 7 & 21076543ba98 & b6789a501234 & 1,4,18,83,233,514,1045,1571,1648,1373,\\
    &    &              &              &    \quad 869,341,73,7\\
    & 8 & 054321ba9876 & 9ab834567012 & 1,11,59,213,579,1216,1920,2216,1823,\\
    &    &              &              &    \quad 1034,386,89,8\\
    & 23 & 543210ba9876 & 9ab345678012 & 1,13,71,207,337,311,153,23\\
    & 25 & 10765432ba98 & 9ab345678012 & 1,24,253,1527,5662,13109,18983,16997,\\
    &    &              &              &    \quad 9166,2836,453,25\\
    & 26 & 10765432ba98 & 789ab1234560 & 1,21,191,933,2561,4008,3573,1735,387,26\\
    & 27 & 10765432ba98 & b6789a012345 & 1,21,191,933,2554,3994,3583,1772,415,27\\
    & 158 & 210876543ba9 & b6789a123450 & 1,24,266,1752,7380,20722,39703,52400,\\
    &    &              &              &    \quad 47388,28667,10969,2301,158\\
    & 163 & 21076543ba98 & b6789a123450 & 1,23,250,1682,7564,23555,51779,80733,\\
    &    &              &              &    \quad 88768,67850,35154,11769,2280,163\\\midrule
    13 & 796 & 321087654cba9 & c789ab1234560 & 1,27,347,2808,15615,62330,\\
    &    &              &              &    \quad 183306,401999,658761,802957,\\
    &    &              &              &    \quad 721035,469418,215528,66010,12044,796\\\bottomrule
  \end{tabular}\label{table:klpolys}
\end{table}

\subsection{Equivalence classes of pairs}
\label{sec:ecs}

Let $\mup{n} = \EPmun{n} \cup \{(x,w):\, w\text{ covers }x\}$ denote the
set of pairs $(x,w)\in S_n\times S_n$ for which $\mu(x,w) > 0$.  Write
$\mupall$ for the union of $\mup{n}$ as $n$ runs over the positive
integers.  The identities in Facts~\ref{fact:symm},~\ref{fact:samel}
and~\ref{fact:compress} allow us to define the following equivalence
relation on the elements of $\mupall$: Two pairs in $\mupall$ are
$\simlspm$-equivalent if they can be connected by a finite chain
consisting of LS-moves, compressions/decompressions and symmetries.
(I.e., $\simlspm$ is the transitive closure of the union of the
relations arising from Facts~\ref{fact:symm},~\ref{fact:samel}
and~\ref{fact:compress}.)

By construction, $\mu[\cdot,\cdot]$ is constant on
$\simlspm$-equivalence classes.  Hence, the number of classes
intersecting $S_m$ for $m\leq n$ gives an upper bound on the size of
$\muset(n)$.  Unfortunately, we have no algorithm (in the precise
sense of the word) for computing the equivalence classes: To show
$(x,w)$ and $(y,v)$ are equivalent, we must provide a chain $(x,w)
\simlspm (x',w') \simlspm \cdots \simlspm (y,v)$ where each successive
pair is connected by either an L-S operator, a compression, a
decompression or a symmetry.  However, we have no bound on how large a
symmetric group we might have to pass through in order to construct
such a chain; we can \emph{always} decompress.  In other words, given
pairs with the same $\mu$-value, we have no effective method for
showing that they are \emph{not} in the same $\simlspm$-equivalence
class.  In light of this problem, we define $\mu$-positive pairs
$(x,w)\in S_m$ and $(x',w')$ in $S_n$ to be
\emph{$\simlspmk{k}$-equivalent} if they can be connected by a chain
that does not pass through $S_{\max(m,n)+k+1}$.  An
\emph{$\nkmin{n}{k}$-minimal} pair is one whose
$\simlspmk{k}$-equivalence class does not intersect $S_m$ with $m <
n$.  The irreducible pairs in $S_n$ with positive $\mu$-value are the
$\nkmin{n}{0}$-minimal pairs.

\begin{table}[ht]
  \centering
  \caption{Coalescence of $\simlspmk{k}$-equivalence classes.}
  \begin{tabular}{@{}crrrrr@{}} \toprule
    & & & \multicolumn{3}{c}{k}\\\cmidrule(l){4-6}
    $n$ & $\mu$ & No. & 0 & 1 & 2 \\\midrule
    9 & 1 & 12 & 3 & s\phantom{+0} & s\phantom{+0}\\\midrule
    10 & 1 & 586 & 31 & s+1 & s+1 \\
       & 4 & 428 & 10 & 3 & 2 \\
       & 5 & 1498 & 27 & 2 & 1\\\midrule
    11 & 1 & 26336 & 419 &s+1 & s+1\\
       & 3 & 2466 & 36 & 2 & 1\\
       & 4 & 5166 & 59 & s+3 & s+1\\
       & 5 & 17052 & 170 & s\phantom{+0}& s\phantom{+0}\\
       & 18 & 16 & 1 & 1 & 1\\
       & 24 & 16 & 1 & 1 & 1\\
       & 28 & 8 & 2 & 2 & 2\\
  \end{tabular}\label{tab:coalesce}
\end{table}

Let $A$ be the $(|\mup{n}|+1)\times (|\mup{n}|+1)$ $0$--$1$ matrix
with the first row and column indexed by a ``sink'' and all other
rows/columns indexed by the elements of $\mup{n}$.  The sink will
identify all pairs in $\mup{n}$ that are not $\nkmin{n}{k}$-minimal.
There is a straightforward algorithm for determining the
$\nkmin{n}{k}$-minimal equivalence classes.
\begin{enumerate}
  \item Pick $k$.  Initialize all entries of $A$ to $0$.
  \item\label{step:bfs} For each pair $(x,w)\in \mup{n}$ (indexing
    row/column $i$), perform a breadth-first search of the members of
    its $\simlspmk{k}$-equivalence class by considering L-S moves,
    symmetries, compressions and decompressions.  (Only allow
    decompressions in the case that the resulting pair lies in $S_m$
    for some $m\leq n+k$.)
  \item For each pair $(y,v)$ (indexing row/column $j$) encountered in
    Step~\ref{step:bfs}, set $A(i,j) = 1$.
  \item If $(x,w)$ is related to a pair in some $S_m$, $m < n$, then
    set $A(i,1) = 1$.
  \item We then compute the connected components using Matlab's
    \texttt{graphconncomp} command.  (Since $A$ may be missing edges
    originating at the sink, we use the `weak' option.)
\end{enumerate}

Table~\ref{tab:coalesce} illustrates how the various equivalence
classes coalesce for $9\leq n\leq 11$ as $k$ ranges from $0$ to $2$.
An $s$ entry (for ``sink'') indicates that some of the pairs are not
$\nkmin{n}{k}$-minimal.  Theorem~\ref{thm:numequiv} is immediate.  We
computed the corresponding $\nkmin{n}{3}$-minimal classes for all
cases except the $\mu=1$, $n=11$ class for which we ran out of memory.
For the computed cases, the $\nkmin{n}{3}$-minimal classes equaled the
$\nkmin{n}{2}$-minimal classes.  Figure~\ref{fig:bothminreps} gives
the Bruhat pictures for (non-canonical) representatives of each
$\nkmin{n}{2}$-minimal class.

\begin{figure}[htbp]
  \centering
  {\scalebox{.4}{\includegraphics{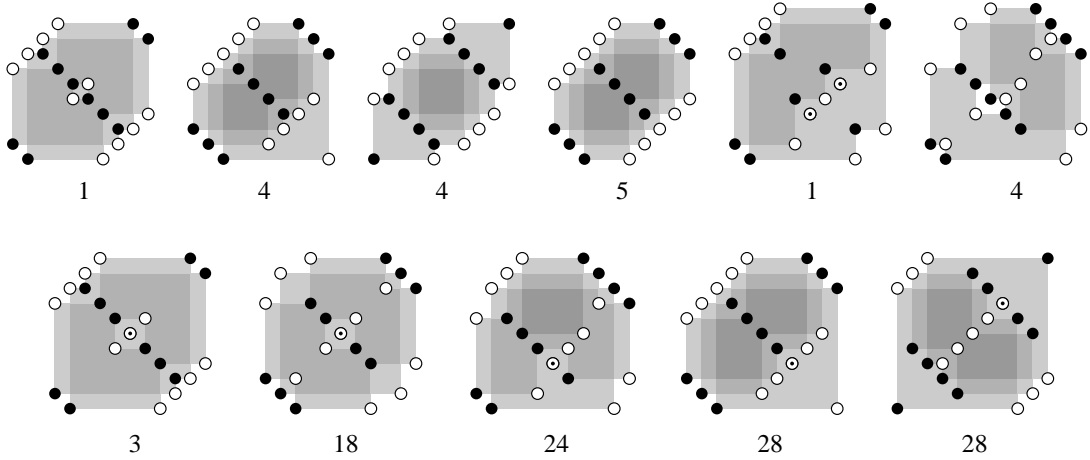}}}
  \caption{Representatives of $\nkmin{n}{2}$-minimal classes.}
  \label{fig:bothminreps}
\end{figure}

We suspect that some of these classes may coalesce further as $k$ is
increased.  However, already at $k=3$ computations become demanding.
For example, consider the $\nkmin{11}{0}$-minimal pair
$x=21076543a98$, $w=6789a123450$.  The size of its
$\simlspmk{k}$-equivalence class grows from 1\,032 to 879\,316 to
331\,361\,376 as $k$ goes from 1 to 2 to 3.

As an example of coalescence, we consider one of the twelve
$\nkmin{9}{0}$-minimal pairs in $S_9$.  Figure~\ref{fig:9reduct}
demonstrates the equality $\ec{216540873}{567812340} = \ec{01}{10}$.
Any chain connecting these two pairs must pass through $S_{10}$.  This
example also serves to illustrate that the Kazhdan-Lusztig polynomials
are not preserved by the L-S operators; $P_{01,10}(q) = 1$ while
  \begin{equation*}
    P_{216540873,567812340}(q) = 1 + 8q + 16q^2 + 11q^3 + q^4.
  \end{equation*}
  \begin{figure}[ht]
    \centering
    {\scalebox{.5}{\includegraphics{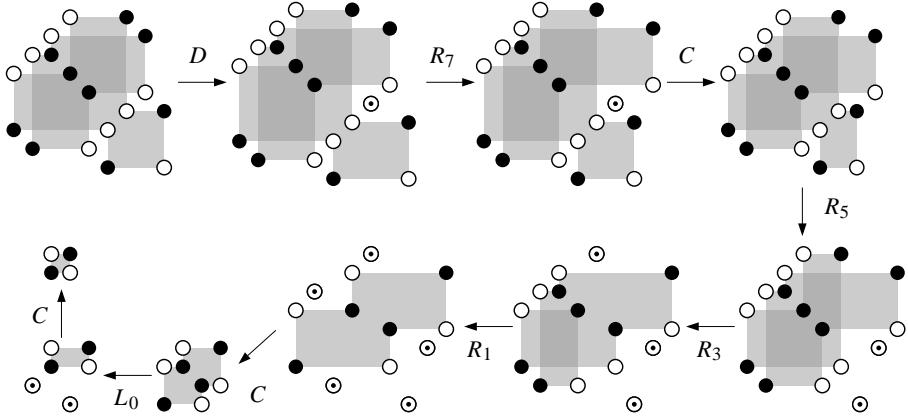}}}
    \caption{Reduction of $(x,w) = (216540873,567812340)$.}
    \label{fig:9reduct}
  \end{figure}

\section{Representatives of equivalence classes}
\label{sec:reps}

Given a composition $\alpha = (\alpha_1,\alpha_2,\ldots,\alpha_k)
\models n$, let $x_\alpha$ be the permutation 
\begin{equation*}
  [n-\alpha_1,n-\alpha_1+1,\ldots,n-1,n-\alpha_1-\alpha_2,
    n-\alpha_1-\alpha_2+1,\ldots, n-\alpha_1-1,\ldots,0,1,\ldots,\alpha_k-1].
\end{equation*}
Let $X_n = \{x_\alpha:\, \alpha\models n\}$.  We define a 
\emph{crosshatch pair} to be a pair $x\leq w$ for which
$xw_0,w\in X_n$.

\begin{conjecture}\label{conj:main}
  Every $\simlspm$-equivalence class contains a crosshatch pair.
\end{conjecture}

In particular, while we conjecture that each $n$-minimal class has a
crosshatch pair, there may only be such pairs in $S_m$ with $m > n$.
Even after factoring out symmetry, such putative representatives are
not unique.  Recall that Figure~\ref{fig:bothminreps} gives
representatives for the various $\nkmin{n}{2}$-minimal equivalence
classes that we have been able to compute.  For five of these classes
(one $n=10$, $\mu=4$ class and the $\mu=1,4,18,24$ classes for
$n=11$), the representative given in that figure is not a crosshatch
pair.  Figure~\ref{fig:chreps} remedies this for four of the classes
by giving crosshatch representatives lying in $S_{m}$ with $m$ equal
to 12 or 13.  The class we were unable to find a crosshatch
representative for is the $n=11$, $\mu=4$ class.  However, given our
above remark about the sizes of $\simlspmk{k}$-equivalence classes, we
do not feel this is a significant mark against
Conjecture~\ref{conj:main}.  The three possibilities are that this
class is not $\nkmin{11}{k}$-minimal for some $k > 3$, that its
smallest crosshatch pair lies in $S_m$ for some $m\geq 15$, or that it
does not contain a crosshatch pair at all.

\begin{figure}[htbp]
  \centering
      {\scalebox{.5}{\includegraphics{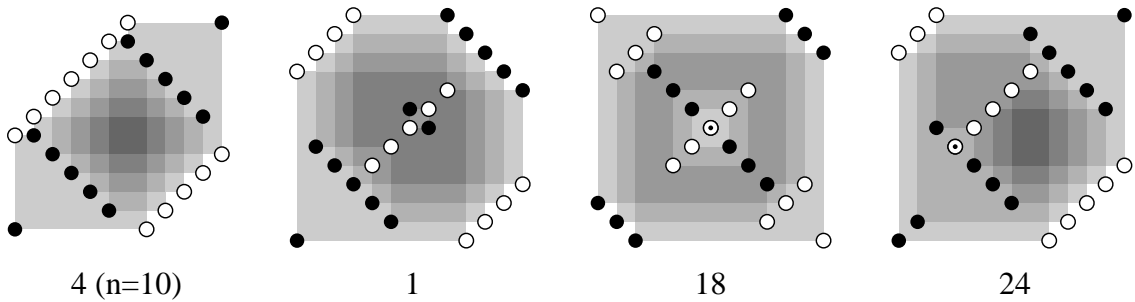}}}
      \caption{Crosshatch representatives.}
      \label{fig:chreps}
\end{figure}

In light of Conjecture~\ref{conj:main}, it is reasonable to ask if
there are simple criteria for the $\mu$-value of a crosshatch pair to be
nonzero.  Or even more ambitiously, to ask for a simple closed formula
for the value of $\mu$ on such an interval.  We note here that Brenti
(along with various coauthors --- see~\cite{brenticm,brentiinc}) has
closed formulas for Kazhdan-Lusztig polynomials based on alternating
sums of paths that might be specialized for this purpose.

It would also be interesting to understand geometrically why such
intervals appear so prevalent among pairs with $\mu$-values greater
than 1; the crosshatch intervals are minimal coset representatives for
certain Richardson varieties with respect to independent partial flag
manifolds~\cite{knutson}.  Of course, everything in this section may
be attributable to working with values of $n$ that are too small.  On
the other hand, crosshatch pairs are relatively rare even for these
small values of $n$.  Of the 1.2 billion extremal pairs in $S_{10}$
only $4\,708$ are crosshatch pairs.

\section{Acknowledgments} 
I would like to thank Allen Knutson for helpful discussions.  The
hash-generation code released by Bob Jenkins~\cite{jenkins} and the
hash table code released by Troy D. Hanson~\cite{uthash} are much more
robust and efficient than anything I could have written.  I am also
grateful to the University Vermont and its Vermont Advanced Computing
Center for generously providing access to its cluster.

\bibliographystyle{plain}
\bibliography{klmu}


\end{document}